\documentclass[11pt]{article}

\usepackage[margin=1.05in]{geometry}
\usepackage{amsmath,amssymb,amsthm,mathtools}
\usepackage{mathrsfs}
\usepackage{enumitem}
\usepackage{hyperref}

\hypersetup{
	colorlinks=true,
	linkcolor=blue,
	citecolor=blue,
	urlcolor=blue
}

\newtheorem{theorem}{Theorem}[section]
\newtheorem{lemma}[theorem]{Lemma}
\newtheorem{proposition}[theorem]{Proposition}

\theoremstyle{definition}

\theoremstyle{remark}
\newtheorem{remark}[theorem]{Remark}

\numberwithin{equation}{section}

\newcommand{\R}{\mathbb{R}}

\newcommand{\norm}[1]{\left\lVert #1 \right\rVert}
\newcommand{\Dcal}{\mathcal D}
\newcommand{\Pcal}{\mathcal P}

\title{A Sharp Mass Threshold for Fractional Choquard Equations with Lower Hardy--Littlewood--Sobolev and \(L^2\)-Critical Terms}

\author{Yongpeng Chen and Zhipeng Yang\thanks{Corresponding author: yangzhipeng326@163.com.}}

\date{}

\AtEndDocument{%
	\par
	\bigskip
	\bigskip

	\noindent
	\textbf{Yongpeng Chen:}\\[0.2em]
	\textsc{School of Science, Guangxi University of Science and Technology, Liuzhou, China}\\[0.3em]
	\textit{E-mail address}: \texttt{yongpengchen@mail.bnu.edu.cn}\par\vspace{1.5em}

	\noindent
	\textbf{Zhipeng Yang}\\[0.2em]
	\textsc{Department of Mathematics, Yunnan Normal University, Kunming, China}\\
	\textsc{Yunnan Key Laboratory of Modern Analytical Mathematics and Applications, Yunnan Normal University, Kunming, China}\\[0.3em]
	\textit{E-mail address}: \texttt{yangzhipeng326@163.com}%
}

\begin{document}

\maketitle

\begin{abstract}
For \(\frac12\le s<1\), we study a fractional Choquard equation with
prescribed mass and nonlinearities at the lower Hardy--Littlewood--Sobolev
and \(L^2\)-critical exponents. The sharp
Hardy--Littlewood--Sobolev and Choquard Gagliardo--Nirenberg inequalities
determine an explicit critical mass \(a_*\). For \(0<a\le a_*\), we compute
the exact infimum of the constrained energy and prove that it is not attained
and that no normalized solution exists. For \(a>a_*\), the energy is
unbounded from below on the mass sphere, while the Pohozaev set is nonempty.
\end{abstract}



\section{Introduction and main results}
\label{sec:intro}

Choquard equations, also called Hartree equations, are nonlocal elliptic
equations whose nonlinear terms involve a Riesz potential. A fractional
model is
\begin{equation}
\label{eq1.1}
(-\Delta)^s u+V(x)u
=
\big(I_\alpha*|u|^p\big)|u|^{p-2}u
\quad\text{in }\R^N,
\end{equation}
where \(0<s<1\), \(\alpha\in(0,N)\), and
\[
I_\alpha(x)=A_{N,\alpha}|x|^{-(N-\alpha)}.
\]
The variational analysis of \eqref{eq1.1} uses the
Hardy--Littlewood--Sobolev inequality together with the fractional Sobolev
embedding.

For the equation with fixed frequency
\[
(-\Delta)^s u+\omega u
=
\big(I_\alpha*|u|^p\big)|u|^{p-2}u
\quad\text{in }\R^N,
\]
the variational range in the local case \(s=1\) is
\[
\frac{N+\alpha}{N}<p<\frac{N+\alpha}{N-2}.
\]
Existence and qualitative properties of ground states in this setting were
studied, among others, in \cite{Lieb1977,MorozVanSchaftingen2013,MorozVanSchaftingen2015}.
The fractional counterpart was developed by d'Avenia, Siciliano and Squassina
\cite{DAveniaSicilianoSquassina2015}. The Pohozaev argument below uses the
regularity estimates of Ambrosio \cite{Ambrosio2025Remarks} and the integration
formulas of Cingolani, Gallo and Tanaka \cite{CingolaniGalloTanaka2024}.
The sharp Hardy--Littlewood--Sobolev inequality of
Lieb \cite{Lieb1983} gives the lower endpoint constant used below.

We are concerned with the normalized problem, where the mass is prescribed and
the frequency is a Lagrange multiplier. The constraint is
\[
\int_{\R^N}|u|^2\,dx=a>0.
\]
This constraint arises in the study of standing waves with conserved mass.
For local nonlinear Schr\"odinger equations, constrained variational methods go
back at least to Jeanjean \cite{Jeanjean1997}. For nonlocal models, Bellazzini,
Jeanjean and Luo \cite{BellazziniJeanjeanLuo2013} treated the
Schr\"odinger--Poisson--Slater equation in \(\R^3\), while Cingolani and
Jeanjean \cite{CingolaniJeanjean2019} studied the planar Schr\"odinger--Poisson
system. For normalized Choquard equations, we refer to Li and Ye
\cite{LiYe2014} and Yuan, Chen and Tang \cite{YuanChenTang2020}. Fractional
normalized Choquard equations and critical variants were studied
in \cite{LiLuo2020,HeRadulescuZou2022,FengHeMeng2023,LanHeMeng2023}. Lower
Hardy--Littlewood--Sobolev critical terms under mass constraints also occur in
\cite{ChenYangTJM2025,ChenKumarYangZhang2026}.

We study the fractional Choquard energy with nonlinearities at two critical
exponents. Set
\[
\Dcal_p(u)=
\int_{\R^N}\big(I_\alpha*|u|^p\big)|u|^p\,dx.
\]
For the dilation that preserves the mass
\begin{equation}
\label{eq1.2}
u_\sigma(x)=\sigma^{\frac N2}u(\sigma x),
\qquad \sigma>0,
\end{equation}
one has
\[
\norm{(-\Delta)^{s/2}u_\sigma}_2^2
=
\sigma^{2s}\norm{(-\Delta)^{s/2}u}_2^2,
\qquad
\Dcal_p(u_\sigma)
=
\sigma^{N(p-1)-\alpha}\Dcal_p(u).
\]
Thus the two relevant endpoint exponents are
\[
q_L=\frac{N+\alpha}{N},
\qquad
 t_c=\frac{N+\alpha+2s}{N}.
\]
The first exponent is the lower Hardy--Littlewood--Sobolev endpoint, for which
\(\Dcal_{q_L}\) is invariant under \eqref{eq1.2}. The second is the
\(L^2\)-critical Choquard exponent, for which \(\Dcal_{t_c}\) has the same
scaling as the kinetic term.

We study
\begin{equation}
\label{eq1.3}
(-\Delta)^s u
=
\lambda u
+
\big(I_\alpha*|u|^{q_L}\big)|u|^{q_L-2}u
+
\big(I_\alpha*|u|^{t_c}\big)|u|^{t_c-2}u
\quad\text{in }\R^N,
\end{equation}
under the constraint \(\norm{u}_2^2=a\). Throughout the paper,
\begin{equation}
\label{eq1.4}
N\ge3,
\qquad
\frac12\le s<1,
\qquad
0<\alpha<N,
\end{equation}
and all functions are real-valued. Define
\[
S(a)=
\left\{
 u\in H^s(\R^N):\norm{u}_2^2=a
\right\}
\]
and
\[
J(u)
=
\frac12\norm{(-\Delta)^{s/2}u}_2^2
-
\frac{1}{2q_L}\Dcal_{q_L}(u)
-
\frac{1}{2t_c}\Dcal_{t_c}(u),
\qquad
u\in S(a).
\]
Critical points of \(J|_{S(a)}\) are weak normalized solutions of
\eqref{eq1.3}.

The analysis involves two sharp constants. The lower endpoint constant is
\begin{equation}
\label{eq1.5}
L_{N,\alpha}
=
\sup_{u\in L^2(\R^N)\setminus\{0\}}
\frac{\Dcal_{q_L}(u)}{\norm{u}_2^{2q_L}},
\end{equation}
and the critical Choquard Gagliardo--Nirenberg constant is
\begin{equation*}
C_{N,\alpha,s}
=
\sup_{u\in H^s(\R^N)\setminus\{0\}}
\frac{\Dcal_{t_c}(u)}
{\norm{u}_2^{2(t_c-1)}\norm{(-\Delta)^{s/2}u}_2^2}.
\end{equation*}
The corresponding mass threshold is
\[
a_*
=
\left(
\frac{t_c}{C_{N,\alpha,s}}
\right)^{\frac{N}{2s+\alpha}}.
\]

Our first result gives the exact energy infimum below the critical mass.

\begin{theorem}
\label{Thm1.1}
Assume \eqref{eq1.4} and let \(0<a<a_*\). Then
\begin{equation*}
\inf_{u\in S(a)}J(u)
=
-
\frac{1}{2q_L}L_{N,\alpha}a^{q_L}.
\end{equation*}
Moreover, this infimum is not attained on \(S(a)\), and problem
\eqref{eq1.3} has no normalized solution with mass \(a\).
\end{theorem}

At the critical mass, the same level is the infimum, but it is not attained.
The proof uses the classification of critical Hardy--Littlewood--Sobolev
optimizers and shows that they do not satisfy the Euler--Lagrange equation
for the critical Choquard Gagliardo--Nirenberg inequality.

\begin{theorem}
\label{Thm1.2}
Assume \eqref{eq1.4} and let \(a=a_*\). Then
\begin{equation*}
\inf_{u\in S(a_*)}J(u)
=
-
\frac{1}{2q_L}L_{N,\alpha}a_*^{q_L}.
\end{equation*}
Moreover, this infimum is not attained on \(S(a_*)\), and problem \eqref{eq1.3}
has no normalized solution with mass \(a_*\).
\end{theorem}

For \(a>a_*\), direct minimization fails. Set
\[
P(u)
=
\norm{(-\Delta)^{s/2}u}_2^2
-
\frac1{t_c}\Dcal_{t_c}(u).
\]

\begin{theorem}
\label{Thm1.3}
Assume \eqref{eq1.4} and let \(a>a_*\). Then
\begin{equation*}
\inf_{u\in S(a)}J(u)=-\infty.
\end{equation*}
Moreover, the Pohozaev set
\begin{equation*}
\Pcal_a=
\left\{
 u\in S(a):P(u)=0
\right\}
\end{equation*}
is nonempty. If \(u\in S(a)\) is a normalized solution of \eqref{eq1.3}, then
\(u\in\Pcal_a\), and its Lagrange multiplier satisfies \(\lambda<0\).
\end{theorem}

\begin{remark}
\label{Rem1.4}
In the present equation, both Choquard terms have coefficient one, and the
mass threshold is \(a_*\). If coefficients are placed in front of these terms,
the coefficient of the term with exponent \(t_c\) changes \(a_*\), whereas
the coefficient of the lower endpoint term changes only the energy infimum.
\end{remark}

\begin{remark}
\label{Rem1.5}
For \(a>a_*\), the energy is unbounded from below on \(S(a)\), so normalized
solutions cannot be obtained by direct minimization. The relation \(\Pcal_a\ne
\emptyset\) alone does not show that \(\Pcal_a\) is a natural constraint for
\(J\). The obstruction is the common scaling factor \(\sigma^{2s}\) of the
kinetic term and the Choquard term with exponent \(t_c\) under \eqref{eq1.2}. The existence of
normalized solutions for all large masses remains open.
\end{remark}

The paper is organized as follows. Section \ref{sec:setting} collects the sharp
endpoint inequalities and the differentiability properties of the constrained
functional. Section \ref{sec:threshold} gives the exact energy infimum,
the Pohozaev identity, and the nonexistence results below and at the threshold.
Section \ref{sec:large} discusses large masses and the Pohozaev constraint.
Section \ref{sec:proofs} proves the main theorems.

\section{Sharp inequalities and functional setting}
\label{sec:setting}

We write \(2_s^*=2N/(N-2s)\).
We shall use the following Hardy--Littlewood--Sobolev inequality.

\begin{lemma}\cite{Lieb1983}
\label{Lem2.1}
Let \(r,t>1\) and
\begin{equation*}
\frac1r+\frac1t+\frac{N-\alpha}{N}=2.
\end{equation*}
Then there exists \(C>0\) such that
\begin{equation*}\int_{\R^N}\int_{\R^N}
\frac{f(x)g(y)}{|x-y|^{N-\alpha}}\,dx\,dy
\le
C\norm{f}_r\norm{g}_t
\end{equation*}
for all \(f\in L^r(\R^N)\) and \(g\in L^t(\R^N)\). Consequently, \(\Dcal_p(u)\) is well defined on \(H^s(\R^N)\) for
\begin{equation*}\frac{N+\alpha}{N}
\le p\le
\frac{N+\alpha}{N-2s}.
\end{equation*}
\end{lemma}

\begin{lemma}
\label{Lem2.2}
The endpoint constant \(L_{N,\alpha}\) defined in \eqref{eq1.5} is finite, and
\begin{equation}
\label{eq2.1}\Dcal_{q_L}(u)
\le
L_{N,\alpha}\norm{u}_2^{2q_L}
\qquad
\text{for all }u\in L^2(\R^N).
\end{equation}
Moreover, the constant \(L_{N,\alpha}\) is attained.
\end{lemma}

\begin{proof}
Let \(r=2N/(N+\alpha)\). Since \(rq_L=2\), the
Hardy--Littlewood--Sobolev inequality gives
\[
\Dcal_{q_L}(u)\le C\norm{|u|^{q_L}}_r^2=C\norm{u}_2^{2q_L},
\qquad u\in L^2(\R^N).
\]
Hence \(L_{N,\alpha}<\infty\), and \eqref{eq2.1} follows from the definition
of \(L_{N,\alpha}\).

For attainment, let \(F\ge0\), \(F\not\equiv0\), be an optimizer in Lieb's
sharp Hardy--Littlewood--Sobolev inequality \cite{Lieb1983} for the exponent
\(r\). Then \(u=F^{1/q_L}\in L^2(\R^N)\) and
\[
\norm{u}_2^{2q_L}=\norm{F}_r^2,
\qquad
\Dcal_{q_L}(u)=
A_{N,\alpha}\int_{\R^N}\int_{\R^N}
\frac{F(x)F(y)}{|x-y|^{N-\alpha}}\,dx\,dy.
\]
Thus \(u\) attains \(L_{N,\alpha}\).
\end{proof}

\begin{lemma}
\label{Lem2.3}
The sharp Choquard Gagliardo--Nirenberg inequality at the \(L^2\)-critical exponent is
\begin{equation}
\label{eq2.2}\Dcal_{t_c}(u)
\le
C_{N,\alpha,s}
\norm{u}_2^{2(t_c-1)}
\norm{(-\Delta)^{s/2}u}_2^2
\qquad
\text{for all }u\in H^s(\R^N).
\end{equation}
Moreover, the sharp constant \(C_{N,\alpha,s}\) is attained by a nonnegative radially decreasing function.
\end{lemma}

\begin{proof}
Let \(\rho=2Nt_c/(N+\alpha)\). Then \(2<\rho<2_s^*\). Lemma
\ref{Lem2.1} and the fractional Gagliardo--Nirenberg inequality yield
\[
\Dcal_{t_c}(u)\le C\norm{u}_{\rho}^{2t_c}
\le
C\norm{(-\Delta)^{s/2}u}_2^2\norm{u}_2^{2(t_c-1)},
\]
because
\[
\frac{N}{s}\left(\frac12-\frac1\rho\right)=\frac1{t_c}.
\]
This gives \(C_{N,\alpha,s}<\infty\), and \eqref{eq2.2} follows from its
definition.

We prove attainment. By the homogeneity and the \(L^2\)-preserving scaling
invariance of the quotient defining \(C_{N,\alpha,s}\), choose a maximizing
sequence \(\{u_n\}\subset H^s(\R^N)\) such that
\[
\norm{u_n}_2=1,
\qquad
\norm{(-\Delta)^{s/2}u_n}_2=1,
\qquad
\Dcal_{t_c}(u_n)\to C_{N,\alpha,s}.
\]
By Schwarz symmetrization and a final normalization, using the Riesz
rearrangement inequality and the fractional P\'olya--Szeg\H{o} inequality
\cite{LiebLoss2001}, we may assume that each \(u_n\) is nonnegative and
radially decreasing. Passing to a subsequence,
\[
u_n\rightharpoonup Q \quad\text{in }H^s(\R^N).
\]
The compact radial embedding \cite{Lions1982,PalatucciPisante2014} gives
\(u_n\to Q\) strongly in \(L^\rho(\R^N)\). Hence
\(|u_n|^{t_c}\to |Q|^{t_c}\) strongly in
\(L^{2N/(N+\alpha)}(\R^N)\), and the Hardy--Littlewood--Sobolev inequality
implies \(\Dcal_{t_c}(u_n)\to\Dcal_{t_c}(Q)\). Therefore
\(\Dcal_{t_c}(Q)=C_{N,\alpha,s}\), so \(Q\not\equiv0\). Since
\(\norm{Q}_2\le1\) and \(\norm{(-\Delta)^{s/2}Q}_2\le1\), the sharp inequality
already proved implies equality in both bounds. Thus \(Q\) attains
\(C_{N,\alpha,s}\). It is nonnegative and radially decreasing by construction.
\end{proof}

\begin{lemma}
\label{Lem2.4}
For every \(u\in H^s(\R^N)\) and every \(\sigma>0\), let \(u_\sigma\) be defined by \eqref{eq1.2}. Then
\begin{equation}
\label{eq2.3}J(u_\sigma)
=
\frac{\sigma^{2s}}2
\left(
\norm{(-\Delta)^{s/2}u}_2^2
-
\frac1{t_c}\Dcal_{t_c}(u)
\right)
-
\frac1{2q_L}\Dcal_{q_L}(u).
\end{equation}
\end{lemma}

\begin{proof}
The scaling \eqref{eq1.2} preserves the \(L^2\)-norm and satisfies
\[
\norm{(-\Delta)^{s/2}u_\sigma}_2^2
=\sigma^{2s}\norm{(-\Delta)^{s/2}u}_2^2,
\qquad
\Dcal_p(u_\sigma)=\sigma^{N(p-1)-\alpha}\Dcal_p(u).
\]
Since \(N(q_L-1)-\alpha=0\) and \(N(t_c-1)-\alpha=2s\), substituting these
relations into the definition of \(J\) gives \eqref{eq2.3}.
\end{proof}

\begin{lemma}
\label{Lem2.5}
The functional \(J\) is of class \(C^1\) on \(H^s(\R^N)\), and its restriction
to \(S(a)\) is of class \(C^1\) on the Hilbert manifold \(S(a)\). If
\(u\in S(a)\) is a critical point of \(J|_{S(a)}\), then there exists
\(\lambda\in\R\) such that \((u,\lambda)\) solves \eqref{eq1.3} weakly.
\end{lemma}

\begin{proof}
Let \(p\in\{q_L,t_c\}\) and
\[
\rho_p=\frac{2Np}{N+\alpha}.
\]
Then \(2\le\rho_p\le 2_s^*\), and hence
\[
H^s(\R^N)\hookrightarrow L^{\rho_p}(\R^N).
\]
Since the map
\[
u\mapsto |u|^p
\]
is \(C^1\) from \(L^{\rho_p}(\R^N)\) into
\(L^{\frac{2N}{N+\alpha}}(\R^N)\), the Hardy--Littlewood--Sobolev inequality
implies that
\[
u\mapsto \Dcal_p(u)
\]
is \(C^1\) on \(H^s(\R^N)\), with
\[
\langle \Dcal_p'(u),v\rangle
=
2p
\int_{\R^N}
\big(I_\alpha*|u|^p\big)|u|^{p-2}uv\,dx .
\]
Therefore \(J\in C^1(H^s(\R^N),\R)\), and
\[
\begin{aligned}
\langle J'(u),v\rangle
&=
\int_{\R^N}
(-\Delta)^{s/2}u
(-\Delta)^{s/2}v\,dx  \\
&\quad
-
\int_{\R^N}
\big(I_\alpha*|u|^{q_L}\big)|u|^{q_L-2}uv\,dx \\
&\quad
-
\int_{\R^N}
\big(I_\alpha*|u|^{t_c}\big)|u|^{t_c-2}uv\,dx .
\end{aligned}
\]
Since \(S(a)\) is a \(C^1\) Hilbert submanifold of \(H^s(\R^N)\), the restriction
\(J|_{S(a)}\) is \(C^1\).

If \(u\in S(a)\) is a critical point of \(J|_{S(a)}\), then by the Lagrange
multiplier rule there exists \(\lambda\in\R\) such that
\[
\langle J'(u),v\rangle
=
\lambda\int_{\R^N}uv\,dx
\qquad
\text{for all }v\in H^s(\R^N).
\]
Equivalently,
\[
\begin{aligned}
\int_{\R^N}
(-\Delta)^{s/2}u
(-\Delta)^{s/2}v\,dx
&=
\lambda\int_{\R^N}uv\,dx \\
&\quad+
\int_{\R^N}
\big(I_\alpha*|u|^{q_L}\big)|u|^{q_L-2}uv\,dx \\
&\quad+
\int_{\R^N}
\big(I_\alpha*|u|^{t_c}\big)|u|^{t_c-2}uv\,dx
\end{aligned}
\]
for all \(v\in H^s(\R^N)\). This is the weak formulation of \eqref{eq1.3}.
\end{proof}

\section{Sharp threshold and nonexistence for small and critical masses}
\label{sec:threshold}

For \(a>0\), set
\begin{equation*}
m(a)=\inf_{u\in S(a)}J(u).
\end{equation*}
Throughout this section we use
\begin{equation*}
a_*=
\left(
\frac{t_c}{C_{N,\alpha,s}}
\right)^{\frac1{t_c-1}}.
\end{equation*}

\begin{proposition}
\label{Pro3.1}
If \(0<a\le a_*\), then
\begin{equation}
\label{eq3.1}J(u)
\ge
\frac12
\left(
1-
\frac{C_{N,\alpha,s}}{t_c}a^{t_c-1}
\right)
\norm{(-\Delta)^{s/2}u}_2^2
-
\frac1{2q_L}L_{N,\alpha}a^{q_L}
\end{equation}
for all \(u\in S(a)\). In particular,
\begin{equation}
\label{eq3.2}m(a)
\ge
-
\frac1{2q_L}L_{N,\alpha}a^{q_L}.
\end{equation}
\end{proposition}

\begin{proof}
Let \(u\in S(a)\). By Lemma \ref{Lem2.2},
\[
\Dcal_{q_L}(u)
\le
L_{N,\alpha}a^{q_L}.
\]
By Lemma \ref{Lem2.3},
\[
\Dcal_{t_c}(u)
\le
C_{N,\alpha,s}a^{t_c-1}
\norm{(-\Delta)^{s/2}u}_2^2.
\]
Therefore
\[
\begin{aligned}
J(u)
&=
\frac12\norm{(-\Delta)^{s/2}u}_2^2
-\frac1{2q_L}\Dcal_{q_L}(u)
-\frac1{2t_c}\Dcal_{t_c}(u)  \\
&\ge
\frac12
\left(
1-
\frac{C_{N,\alpha,s}}{t_c}a^{t_c-1}
\right)
\norm{(-\Delta)^{s/2}u}_2^2
-
\frac1{2q_L}L_{N,\alpha}a^{q_L}.
\end{aligned}
\]
This proves \eqref{eq3.1}. Since \(a\le a_*\), the coefficient of the kinetic
term is nonnegative, and taking the infimum over \(S(a)\) gives \eqref{eq3.2}.
\end{proof}

\begin{lemma}
\label{Lem3.2}
For every \(a>0\),
\begin{equation*}
m(a)
\le
-
\frac1{2q_L}L_{N,\alpha}a^{q_L}.
\end{equation*}
\end{lemma}

\begin{proof}
By Lemma \ref{Lem2.2}, the sharp lower endpoint constant \(L_{N,\alpha}\) is attained by some \(U\in L^2(\R^N)\), \(U\not\equiv0\). Replacing \(U\) by
\[
\frac{\sqrt a\,U}{\norm{U}_2},
\]
we may assume
\[
\norm{U}_2^2=a
\]
and
\[
\Dcal_{q_L}(U)=L_{N,\alpha}a^{q_L}.
\]

Since \(C_c^\infty(\R^N)\) is dense in \(L^2(\R^N)\), there exists a sequence
\[
\varphi_n\in C_c^\infty(\R^N)
\]
such that
\[
\varphi_n\to U
\quad\text{in }L^2(\R^N).
\]
Set
\[
v_n=
\frac{\sqrt a\,\varphi_n}{\norm{\varphi_n}_2}.
\]
Then
\[
v_n\in C_c^\infty(\R^N)\cap S(a),
\qquad
v_n\to U
\quad\text{in }L^2(\R^N).
\]
By the Hardy--Littlewood--Sobolev inequality,
\[
\Dcal_{q_L}(v_n)\to \Dcal_{q_L}(U)
=
L_{N,\alpha}a^{q_L}.
\]

For \(\sigma>0\), define
\[
(v_n)_\sigma(x)=\sigma^{\frac N2}v_n(\sigma x).
\]
Then \((v_n)_\sigma\in S(a)\). By Lemma \ref{Lem2.4},
\[
J((v_n)_\sigma)
=
\frac{\sigma^{2s}}2
\left(
\norm{(-\Delta)^{s/2}v_n}_2^2
-
\frac1{t_c}\Dcal_{t_c}(v_n)
\right)
-
\frac1{2q_L}\Dcal_{q_L}(v_n).
\]
For each fixed \(n\), letting \(\sigma\to0^+\), we get
\[
\lim_{\sigma\to0^+}J((v_n)_\sigma)
=
-\frac1{2q_L}\Dcal_{q_L}(v_n).
\]
Hence
\[
m(a)
\le
-\frac1{2q_L}\Dcal_{q_L}(v_n)
\]
for every \(n\). Passing to the limit as \(n\to\infty\), we obtain
\[
m(a)
\le
-\frac1{2q_L}L_{N,\alpha}a^{q_L}.
\]
\end{proof}

\begin{proposition}
\label{Pro3.3}
For every \(0<a\le a_*\),
\begin{equation}
\label{eq3.3}m(a)
=
-
\frac1{2q_L}L_{N,\alpha}a^{q_L}.
\end{equation}
If \(0<a<a_*\), then \(m(a)\) is not attained. If \(a=a_*\), every minimizer must attain equality in both sharp inequalities \eqref{eq2.1} and \eqref{eq2.2}.
\end{proposition}

\begin{proof}
The identity \eqref{eq3.3} follows from Proposition \ref{Pro3.1} and Lemma \ref{Lem3.2}.

Assume first that \(0<a<a_*\), and suppose, for a contradiction, that \(m(a)\) is attained by some \(u\in S(a)\). Then
\[
J(u)=m(a)
=
-\frac1{2q_L}L_{N,\alpha}a^{q_L}.
\]
From the proof of Proposition \ref{Pro3.1}, we have
\[
J(u)
\ge
\frac12
\left(
1-
\frac{C_{N,\alpha,s}}{t_c}a^{t_c-1}
\right)
\norm{(-\Delta)^{s/2}u}_2^2
-
\frac1{2q_L}L_{N,\alpha}a^{q_L}.
\]
Since \(0<a<a_*\),
\[
1-
\frac{C_{N,\alpha,s}}{t_c}a^{t_c-1}>0.
\]
Therefore
\[
\norm{(-\Delta)^{s/2}u}_2^2=0.
\]
Thus \(u\) is constant in \(\R^N\). Since \(u\in L^2(\R^N)\), it follows that \(u=0\), contradicting
\[
\norm{u}_2^2=a>0.
\]
Hence \(m(a)\) is not attained for \(0<a<a_*\).

Now let \(a=a_*\), and assume that \(u\in S(a_*)\) is a minimizer. Then
\[
J(u)
=
-\frac1{2q_L}L_{N,\alpha}a_*^{q_L}.
\]
On the other hand,
\[
J(u)
=
\frac12\norm{(-\Delta)^{s/2}u}_2^2
-\frac1{2t_c}\Dcal_{t_c}(u)
-\frac1{2q_L}\Dcal_{q_L}(u).
\]
By Lemmas \ref{Lem2.2} and \ref{Lem2.3},
\[
\Dcal_{q_L}(u)
\le
L_{N,\alpha}a_*^{q_L}
\]
and
\[
\Dcal_{t_c}(u)
\le
C_{N,\alpha,s}a_*^{t_c-1}
\norm{(-\Delta)^{s/2}u}_2^2
=
t_c\norm{(-\Delta)^{s/2}u}_2^2.
\]
Consequently,
\[
\frac12\norm{(-\Delta)^{s/2}u}_2^2
-\frac1{2t_c}\Dcal_{t_c}(u)
\ge0
\]
and
\[
-\frac1{2q_L}\Dcal_{q_L}(u)
\ge
-\frac1{2q_L}L_{N,\alpha}a_*^{q_L}.
\]
Since their sum is exactly
\[
-\frac1{2q_L}L_{N,\alpha}a_*^{q_L},
\]
both inequalities must be equalities. Therefore
\[
\Dcal_{q_L}(u)
=
L_{N,\alpha}a_*^{q_L}
\]
and
\[
\Dcal_{t_c}(u)
=
C_{N,\alpha,s}a_*^{t_c-1}
\norm{(-\Delta)^{s/2}u}_2^2.
\]
Thus \(u\) attains equality in both sharp inequalities \eqref{eq2.1} and \eqref{eq2.2}.
\end{proof}

\begin{remark}
\label{Rem3.4}
By Proposition \ref{Pro3.3}, any minimizer at \(a=a_*\) must optimize both
the lower endpoint Hardy--Littlewood--Sobolev inequality and the sharp
Choquard Gagliardo--Nirenberg inequality. To prove nonattainment, it therefore
suffices to show that no such common optimizer exists.
\end{remark}

\subsection{The Pohozaev identity and nonexistence below the critical mass}
In this subsection, by a normalized solution of \eqref{eq1.3} with mass \(a\), we mean a weak solution
\(u\in S(a)\) obtained as a critical point of \(J|_{S(a)}\).

\begin{lemma}
\label{Lem3.5}
Let \(u\in S(a)\) be a normalized solution of \eqref{eq1.3}. Then
\begin{equation*}
P(u)=0,
\end{equation*}
where
\begin{equation*}
P(u)
=
\norm{(-\Delta)^{s/2}u}_2^2
-
\frac1{t_c}\Dcal_{t_c}(u).
\end{equation*}
Equivalently,
\begin{equation}
\label{eq3.4}\norm{(-\Delta)^{s/2}u}_2^2
=
\frac1{t_c}\Dcal_{t_c}(u).
\end{equation}
\end{lemma}

\begin{proof}
We first establish the regularity needed for the Pohozaev identity. For
\(p\in\{q_L,t_c\}\), put
\[
V_p=I_\alpha*|u|^p,
\qquad
f_p(z)=|z|^{p-2}z,
\qquad
m=\frac{2N}{N+\alpha},
\qquad
m'=\frac{2N}{N-\alpha}.
\]
Since \(pm\in[2,2_s^*]\), the Sobolev and
Hardy--Littlewood--Sobolev inequalities give
\[
|u|^p\in L^m(\R^N),
\qquad
V_p\in L^{m'}(\R^N),
\qquad
\Dcal_p(u)<\infty.
\]
Moreover, \(V_{q_L}+V_{t_c}>0\) a.e., because \(u\ne0\).
Define the measurable coefficients
\[
H=f_{q_L}(u)+f_{t_c}(u),
\qquad
K=\frac{V_{q_L}f_{q_L}(u)+V_{t_c}f_{t_c}(u)}
{V_{q_L}+V_{t_c}}.
\]
They satisfy
\[
Hu=|u|^{q_L}+|u|^{t_c},
\qquad
(I_\alpha*(Hu))K=V_{q_L}f_{q_L}(u)+V_{t_c}f_{t_c}(u),
\]
and
\[
|H|+|K|\le2\bigl(|u|^{q_L-1}+|u|^{t_c-1}\bigr).
\]
The identities
\[
(q_L-1)\frac{2N}{\alpha}=2,
\qquad
(t_c-1)\frac{2N}{\alpha+2s}=2
\]
imply that
\[
H,K\in L^{\frac{2N}{\alpha}}(\R^N)
+L^{\frac{2N}{\alpha+2s}}(\R^N).
\]
Thus \eqref{eq1.3} can be written as
\[
(-\Delta)^su+u=(I_\alpha*(Hu))K+(\lambda+1)u.
\]
Applying \cite[Proposition~3]{Ambrosio2025Remarks}, with positive linear
coefficient \(1\) and bounded potential \(\lambda+1\), we obtain
\[
u\in L^r(\R^N)
\qquad\text{for every }2\le r<\frac N\alpha\,2_s^*.
\]
Since \(t_c<2_s^*\), choose
\[
\max\left\{2,\frac{Nt_c}{\alpha}\right\}<r<\frac N\alpha\,2_s^*.
\]
Then \(r/p>N/\alpha\) for both powers. Splitting the Riesz kernel into
its restrictions to \(B_1\) and \(\R^N\setminus B_1\), H\"older's
inequality yields
\[
\norm{V_p}_\infty
\le C\bigl(\norm{|u|^p}_{r/p}+\norm{|u|^p}_m\bigr)<\infty.
\]
Indeed, the kernel belongs to \(L^{(r/p)'}\) on the unit ball and to
\(L^{m'}\) on its complement.

Write \((-\Delta)^su+u=g\), where
\[
g=(\lambda+1)u+\sum_{p\in\{q_L,t_c\}}V_pf_p(u).
\]
For \(b=\max\{1,t_c-1\}\), we have
\[
|g|\le C(1+|u|^b),
\qquad
1\le b<2_s^*-1.
\]
Let \(\mathcal B_s\) be the positive convolution kernel of
\( ((-\Delta)^s+1)^{-1}\). It satisfies
\[
\norm{\mathcal B_s}_1=1,
\qquad
\mathcal B_s\in L^\ell(\R^N)
\quad\text{for }1\le\ell<\frac{N}{N-2s}.
\]
The resolvent comparison used in the proof of
\cite[Proposition~4]{Ambrosio2025Remarks}, obtained from the distributional
Kato inequality, gives
\[
|u|\le\mathcal B_s*|g|
\le C+C\mathcal B_s*(|u|^b\mathbf1_E),
\qquad E=\{x:|u(x)|>1\}.
\]
The set \(E\) has finite measure. Starting with \(r_0=2_s^*\), Young's
inequality improves \(u\mathbf1_E\in L^r\) to
\(u\mathbf1_E\in L^{r_1}\) whenever
\[
0<\frac1{r_1}<\frac1r,
\qquad
\frac1{r_1}>\frac br-\frac{2s}{N}.
\]
This iteration reaches an exponent satisfying \(r/b>N/(2s)\) after
finitely many steps. In fact,
\[
\delta=\frac{2s}{N}-\frac{b-1}{r_0}>0,
\]
and, as long as \(b/r\ge2s/N\), the lower bound for \(1/r_1\) is at
most \(1/r-\delta\). The reciprocal exponent can therefore be decreased
by \(\delta/2\) at each such step. Once \(r/b>N/(2s)\),
\(\mathcal B_s\in L^{(r/b)'}\), and the preceding comparison implies
\(u\in L^\infty(\R^N)\).

Since \(s\ge1/2\), the Schauder estimates in
\cite[Proposition~5]{Ambrosio2025Remarks} first give
\(u\in C^{0,1/2}(\R^N)\). The potentials \(V_p\) have the same
H\"older regularity. To see this, choose
\(\eta\in C_c^\infty(\R^N)\), equal to one near the origin, and write
\[
V_p=(\eta I_\alpha)*|u|^p+((1-\eta)I_\alpha)*|u|^p.
\]
The first convolution has the same H\"older regularity because
\(\eta I_\alpha\in L^1\) and \(|u|^p\in C^{0,1/2}\). The second convolution
belongs to \(W^{1,\infty}\), since the kernel away from the origin and its first
derivatives lie in \(L^{m'}\), while \(|u|^p\in L^m\).
Both \(f_{q_L}\) and \(f_{t_c}\) are \(\alpha/N\)-H\"older
continuous on bounded intervals. Hence the right-hand side of
\eqref{eq1.3} belongs to \(C^{0,\alpha/(2N)}(\R^N)\). A further
Schauder estimate gives
\[
u\in C^{1,\beta}(\R^N)
\quad\text{for some }\quad
2s-1<\beta<\min\left\{1,2s-1+\frac{\alpha}{2N}\right\}.
\]
In particular, \(1+\beta>2s\), the fractional Laplacian is pointwise
well defined, and \eqref{eq1.3} holds pointwise. This argument applies to
every \(\lambda\in\R\).

We now apply the integration formulas of Cingolani, Gallo and Tanaka
\cite[Propositions~6.3 and~6.5]{CingolaniGalloTanaka2024}. Their hypotheses
hold because \(u\in H^s\cap C^{1,\beta}\cap L^\infty\),
\(1+\beta>2s\), and, for each \(p\in\{q_L,t_c\}\),
\[
|u|^p\in L^\infty\cap\operatorname{Lip}_{\mathrm{loc}},
\qquad
V_p|u|^p\in L^1,
\qquad
V_p|\nabla(|u|^p)|\in L^1_{\mathrm{loc}}.
\]
Boundedness also gives the weighted integrability at infinity required in
\cite[Proposition~6.3]{CingolaniGalloTanaka2024}.
Choose \(\chi\in C_c^\infty(\R^N)\), equal to one on \(B_1\) and
zero outside \(B_2\), and set \(X_R(x)=\chi(x/R)x\). For \(x\ne y\),
put
\[
\mathcal K_{\kappa,R}(x,y)
=\frac{\operatorname{div}X_R(x)+\operatorname{div}X_R(y)}2
-\frac\kappa2\frac{(X_R(x)-X_R(y))\cdot(x-y)}{|x-y|^2}.
\]
Let \(c_{N,s}\) be normalized by
\[
\norm{(-\Delta)^{s/2}u}_2^2
=\frac{c_{N,s}}2\iint_{\R^{2N}}
\frac{|u(x)-u(y)|^2}{|x-y|^{N+2s}}\,dx\,dy.
\]
Applying the two integration formulas to \eqref{eq1.3}, separately for
each Choquard term, gives
\[
\begin{aligned}
&\frac{c_{N,s}}2\iint_{\R^{2N}}
\frac{|u(x)-u(y)|^2}{|x-y|^{N+2s}}
\mathcal K_{N+2s,R}(x,y)\,dx\,dy\\
&\quad=\frac\lambda2\int_{\R^N}\operatorname{div}X_R\,u^2\,dx\\
&\qquad\quad+\sum_{p\in\{q_L,t_c\}}\frac1p
\iint_{\R^{2N}}I_\alpha(x-y)|u(x)|^p|u(y)|^p
\mathcal K_{N-\alpha,R}(x,y)\,dx\,dy.
\end{aligned}
\]
The factors \(1/p\) follow from
\(\nabla(|u|^p)=p|u|^{p-2}u\nabla u\).
The Lipschitz seminorms of \(X_R\) are uniformly bounded, so the two
kernels in this identity are uniformly bounded. For fixed \(x\ne y\),
\[
\operatorname{div}X_R(x)\longrightarrow N,
\qquad
\mathcal K_{N+2s,R}(x,y)\longrightarrow\frac{N-2s}{2},
\qquad
\mathcal K_{N-\alpha,R}(x,y)\longrightarrow\frac{N+\alpha}{2}.
\]
Dominated convergence, using \(u\in H^s\) and
\(\Dcal_{q_L}(u)+\Dcal_{t_c}(u)<\infty\), therefore yields
\[
\frac{N-2s}{2}\norm{(-\Delta)^{s/2}u}_2^2
=
\frac N2\lambda a
+
\frac{N+\alpha}{2q_L}\Dcal_{q_L}(u)
+
\frac{N+\alpha}{2t_c}\Dcal_{t_c}(u).
\]
On the other hand, testing \eqref{eq1.3} by \(u\) gives
\[
\norm{(-\Delta)^{s/2}u}_2^2
=
\lambda a+\Dcal_{q_L}(u)+\Dcal_{t_c}(u).
\]
Multiplying the last identity by \(N/2\) and subtracting the Pohozaev
identity, we obtain
\[
s\norm{(-\Delta)^{s/2}u}_2^2
=
\left(\frac N2-\frac{N+\alpha}{2q_L}\right)\Dcal_{q_L}(u)
+
\left(\frac N2-\frac{N+\alpha}{2t_c}\right)\Dcal_{t_c}(u).
\]
The first coefficient is zero and the second is \(s/t_c\). Hence
\eqref{eq3.4} follows, and \(P(u)=0\).
\end{proof}

\begin{proposition}
\label{Pro3.6}
If \(0<a<a_*\), then problem \eqref{eq1.3} has no normalized solution with mass \(a\).
\end{proposition}

\begin{proof}
Suppose, for a contradiction, that \(u\in S(a)\) is a normalized solution of \eqref{eq1.3}. By Lemma \ref{Lem3.5},
\[
\norm{(-\Delta)^{s/2}u}_2^2
=
\frac1{t_c}\Dcal_{t_c}(u).
\]
On the other hand, by the sharp inequality \eqref{eq2.2},
\[
\Dcal_{t_c}(u)
\le
C_{N,\alpha,s}
\norm{u}_2^{2(t_c-1)}
\norm{(-\Delta)^{s/2}u}_2^2
=
C_{N,\alpha,s}
a^{t_c-1}
\norm{(-\Delta)^{s/2}u}_2^2.
\]
Hence
\[
\norm{(-\Delta)^{s/2}u}_2^2
\le
\frac{C_{N,\alpha,s}}{t_c}
a^{t_c-1}
\norm{(-\Delta)^{s/2}u}_2^2.
\]
Since \(u\in S(a)\) and \(a>0\), one has \(u\not\equiv0\), and therefore
\[
\norm{(-\Delta)^{s/2}u}_2^2>0.
\]
It follows that
\[
1
\le
\frac{C_{N,\alpha,s}}{t_c}a^{t_c-1}.
\]
Equivalently,
\[
a\ge
\left(
\frac{t_c}{C_{N,\alpha,s}}
\right)^{\frac1{t_c-1}}
=
a_*,
\]
which contradicts \(0<a<a_*\).
\end{proof}

\begin{lemma}
\label{Lem3.7}
If \(u\in S(a)\) is a normalized solution of \eqref{eq1.3}, then
\begin{equation*}
a\ge a_*.
\end{equation*}
Moreover, if \(a=a_*\), then \(u\) is an optimizer for the sharp inequality \eqref{eq2.2}.
\end{lemma}

\begin{proof}
By Lemma \ref{Lem3.5} and \eqref{eq2.2},
\[
\norm{(-\Delta)^{s/2}u}_2^2
=\frac1{t_c}\Dcal_{t_c}(u)
\le
\frac{C_{N,\alpha,s}}{t_c}a^{t_c-1}
\norm{(-\Delta)^{s/2}u}_2^2.
\]
Since \(u\not\equiv0\), the kinetic term is positive. Hence
\(a\ge a_*\). If \(a=a_*\), equality must hold in \eqref{eq2.2}, and \(u\) is
an optimizer.
\end{proof}

\begin{lemma}
\label{Lem3.8}
Let \(u\in S(a)\) be a normalized solution of \eqref{eq1.3}. Then its Lagrange multiplier satisfies
\begin{equation}
\label{eq3.5}\lambda a
=
-
\Dcal_{q_L}(u)
-
\left(1-\frac1{t_c}\right)\Dcal_{t_c}(u).
\end{equation}
In particular,
\begin{equation*}
\lambda<0.
\end{equation*}
\end{lemma}

\begin{proof}
Testing \eqref{eq1.3} by \(u\) gives
\begin{equation}
\label{eq3.6}
\norm{(-\Delta)^{s/2}u}_2^2
=
\lambda\norm{u}_2^2+
\Dcal_{q_L}(u)+\Dcal_{t_c}(u).
\end{equation}
Combining \eqref{eq3.6} with \(\norm{u}_2^2=a\) and Lemma \ref{Lem3.5} gives
\eqref{eq3.5}. Since \(u\not\equiv0\), the positivity of the Riesz kernel gives
\(\Dcal_{q_L}(u)>0\), while \(\Dcal_{t_c}(u)\ge0\) and \(t_c>1\). Thus the
right-hand side of \eqref{eq3.5} is negative, and \(\lambda<0\).
\end{proof}

\subsection{The critical mass case}
\begin{proposition}
\label{Pro3.9}
At the critical mass,
\begin{equation}
\label{eq3.7}m(a_*)
=
-\frac1{2q_L}L_{N,\alpha}a_*^{q_L}.
\end{equation}
Moreover, every normalized solution of \eqref{eq1.3} with mass \(a_*\) is an optimizer for \eqref{eq2.2}.
\end{proposition}

\begin{proof}
The identity \eqref{eq3.7} is exactly Proposition \ref{Pro3.3} with \(a=a_*\).

Let \(u\in S(a_*)\) be a normalized solution of \eqref{eq1.3}. By Lemma \ref{Lem3.7}, \(u\) is an optimizer for the sharp inequality \eqref{eq2.2}. This proves the second assertion.
\end{proof}

\begin{lemma}
\label{Lem3.10}
Let \(u\in S(a_*)\) be an optimizer for \eqref{eq2.2}. Then \(u\) satisfies
\begin{equation}
\label{eq3.8}\int_{\R^N}
(-\Delta)^{s/2}u(-\Delta)^{s/2}v\,dx
+
\frac{t_c-1}{a_*}
\norm{(-\Delta)^{s/2}u}_2^2
\int_{\R^N}uv\,dx
=
\int_{\R^N}
\big(I_\alpha*|u|^{t_c}\big)|u|^{t_c-2}uv\,dx
\end{equation}
for all \(v\in H^s(\R^N)\).
\end{lemma}

\begin{proof}
Since \(u\) is an optimizer for \eqref{eq2.2}, it is a critical point of the quotient
\[
\frac{\Dcal_{t_c}(w)}
{\norm{w}_2^{2(t_c-1)}
\norm{(-\Delta)^{s/2}w}_2^2},
\qquad w\in H^s(\R^N)\setminus\{0\}.
\]
Thus, for every \(v\in H^s(\R^N)\),
\[
\frac{\langle \Dcal_{t_c}'(u),v\rangle}{\Dcal_{t_c}(u)}
=
(t_c-1)
\frac{2\int_{\R^N}uv\,dx}{\norm{u}_2^2}
+
\frac{
2\int_{\R^N}
(-\Delta)^{s/2}u(-\Delta)^{s/2}v\,dx
}{
\norm{(-\Delta)^{s/2}u}_2^2
}.
\]
Since
\[
\langle \Dcal_{t_c}'(u),v\rangle
=
2t_c
\int_{\R^N}
\big(I_\alpha*|u|^{t_c}\big)|u|^{t_c-2}uv\,dx,
\]
we get
\[
\begin{aligned}
\frac{t_c}{\Dcal_{t_c}(u)}
\int_{\R^N}
\big(I_\alpha*|u|^{t_c}\big)|u|^{t_c-2}uv\,dx
&=
\frac{t_c-1}{a_*}
\int_{\R^N}uv\,dx  \\
&\quad+
\frac{
1
}{
\norm{(-\Delta)^{s/2}u}_2^2
}
\int_{\R^N}
(-\Delta)^{s/2}u(-\Delta)^{s/2}v\,dx .
\end{aligned}
\]
Because \(u\in S(a_*)\) attains \eqref{eq2.2}, and
\[
C_{N,\alpha,s}a_*^{t_c-1}=t_c,
\]
we have
\[
\Dcal_{t_c}(u)
=
t_c
\norm{(-\Delta)^{s/2}u}_2^2.
\]
Substituting this identity into the previous equality gives \eqref{eq3.8}.
\end{proof}

\begin{proposition}
\label{Pro3.11}
Let \(u\in S(a_*)\) be a normalized solution of \eqref{eq1.3}. Then there exists a constant \(\kappa>0\) such that
\begin{equation*}
\big(I_\alpha*|u|^{q_L}\big)|u|^{q_L-2}u
=
\kappa u
\qquad
\text{in }H^{-s}(\R^N).
\end{equation*}
Equivalently,
\begin{equation}
\label{eq3.9}\int_{\R^N}
\big(I_\alpha*|u|^{q_L}\big)|u|^{q_L-2}uv\,dx
=
\kappa\int_{\R^N}uv\,dx
\end{equation}
for all \(v\in H^s(\R^N)\).
\end{proposition}

\begin{proof}
By Proposition \ref{Pro3.9}, \(u\) is an optimizer for \eqref{eq2.2}. Hence Lemma \ref{Lem3.10} gives
\begin{equation}
\label{eq3.10}\int_{\R^N}
(-\Delta)^{s/2}u(-\Delta)^{s/2}v\,dx
+
\frac{t_c-1}{a_*}
\norm{(-\Delta)^{s/2}u}_2^2
\int_{\R^N}uv\,dx
=
\int_{\R^N}
\big(I_\alpha*|u|^{t_c}\big)|u|^{t_c-2}uv\,dx .
\end{equation}
On the other hand, since \(u\) is a weak solution of \eqref{eq1.3}, there exists \(\lambda\in\R\) such that
\begin{equation}
\label{eq3.11}\begin{aligned}
\int_{\R^N}
(-\Delta)^{s/2}u(-\Delta)^{s/2}v\,dx
&=
\lambda\int_{\R^N}uv\,dx  \\
&\quad+
\int_{\R^N}
\big(I_\alpha*|u|^{q_L}\big)|u|^{q_L-2}uv\,dx  \\
&\quad+
\int_{\R^N}
\big(I_\alpha*|u|^{t_c}\big)|u|^{t_c-2}uv\,dx
\end{aligned}
\end{equation}
for all \(v\in H^s(\R^N)\). Subtracting \eqref{eq3.10} from \eqref{eq3.11}, we obtain
\[
\int_{\R^N}
\big(I_\alpha*|u|^{q_L}\big)|u|^{q_L-2}uv\,dx
=
-\left(
\lambda+
\frac{t_c-1}{a_*}
\norm{(-\Delta)^{s/2}u}_2^2
\right)
\int_{\R^N}uv\,dx .
\]
Set
\[
\kappa
=
-\lambda-
\frac{t_c-1}{a_*}
\norm{(-\Delta)^{s/2}u}_2^2 .
\]
Then \eqref{eq3.9} holds.

Taking \(v=u\) in \eqref{eq3.9}, we get
\[
\Dcal_{q_L}(u)=\kappa a_*.
\]
Since \(u\not\equiv0\) and the Riesz kernel is positive,
\[
\Dcal_{q_L}(u)>0.
\]
Therefore \(\kappa>0\).
\end{proof}

\begin{lemma}
\label{Lem3.12}
Every real-valued optimizer of \eqref{eq2.2} has a fixed sign.
\end{lemma}

\begin{proof}
Let \(u\) be a real-valued optimizer of \eqref{eq2.2}. Then
\[
\Dcal_{t_c}(|u|)
=
\Dcal_{t_c}(u),
\qquad
\norm{|u|}_2=\norm{u}_2.
\]
Moreover, the Gagliardo representation gives
\[
\norm{(-\Delta)^{s/2}|u|}_2
\le
\norm{(-\Delta)^{s/2}u}_2.
\]
If the last inequality were strict, the quotient defining \(C_{N,\alpha,s}\) would be strictly larger at \(|u|\) than at \(u\), which is impossible. Hence equality holds.

Equality of the Gagliardo seminorms implies equality in
\[
\bigl||u(x)|-|u(y)|\bigr|^2
\le
|u(x)-u(y)|^2
\]
for almost every pair \((x,y)\in\R^N\times\R^N\). Thus
\[
u(x)u(y)\ge0
\]
for almost every pair \((x,y)\). Therefore either \(u\ge0\) a.e. in \(\R^N\) or \(u\le0\) a.e. in \(\R^N\).
\end{proof}

\begin{lemma}
\label{Lem3.13}
Let \(u\in S(a_*)\), \(u\ge0\), be an optimizer for
\eqref{eq2.2}. Suppose that
\begin{equation}
\label{eq3.12}
\big(I_\alpha*u^{q_L}\big)u^{q_L-1}
=
\kappa u
\qquad
\text{in }H^{-s}(\R^N)
\end{equation}
for some \(\kappa>0\). Then there exist constants \(A>0\), \(\rho>0\), and a point \(x_0\in\R^N\) such that
\begin{equation}
\label{eq3.13}
u(x)
=
A
\left(
\frac{\rho}{\rho^2+|x-x_0|^2}
\right)^{\frac N2}
\qquad
\text{for a.e. }x\in\R^N.
\end{equation}
\end{lemma}

\begin{proof}
Since \(u\) is an optimizer for \eqref{eq2.2}, Lemma \ref{Lem3.10} gives
\[
(-\Delta)^s u+\beta u
=
\big(I_\alpha*u^{t_c}\big)u^{t_c-1}
\quad\text{in }H^{-s}(\R^N),
\qquad
\beta=
\frac{t_c-1}{a_*}\norm{(-\Delta)^{s/2}u}_2^2>0.
\]
The right-hand side is nonnegative and nontrivial. Since the resolvent
\(((-\Delta)^s+\beta)^{-1}\) has a strictly positive convolution kernel, \(u>0\) a.e. in \(\R^N\).

Set \(F=u^{q_L}\). Since \(u\in L^2(\R^N)\) and \(q_L=(N+\alpha)/N\),
\(F\in L^{2N/(N+\alpha)}(\R^N)\). The Hardy--Littlewood--Sobolev inequality
implies \(I_\alpha*F\in L^{2N/(N-\alpha)}(\R^N)\), and
\(u^{q_L-1}\in L^{2N/\alpha}(\R^N)\). Hence the left-hand side of
\eqref{eq3.12} belongs to \(L^2(\R^N)\), so \eqref{eq3.12} holds a.e.
Dividing by \(u^{q_L-1}\) yields
\[
I_\alpha*F=
\kappa u^{2-q_L}=
\kappa F^{\frac{N-\alpha}{N+\alpha}}
\quad\text{a.e. in }\R^N.
\]
After multiplying \(F\) by a suitable positive constant, this is the critical
Hardy--Littlewood--Sobolev integral equation. By the classification theorem of
Chen, Li and Ou \cite{ChenLiOu2006},
\[
F(x)
=
B
\left(
\frac{\rho}{\rho^2+|x-x_0|^2}
\right)^{\frac{N+\alpha}{2}}
\]
for some \(B>0\), \(\rho>0\), and \(x_0\in\R^N\). Since
\(u=F^{1/q_L}\), \eqref{eq3.13} follows.
\end{proof}

\begin{lemma}
\label{Lem3.14}
Let
\begin{equation}
\label{eq3.14}U(x)
=
A
\left(
\frac{\rho}{\rho^2+|x-x_0|^2}
\right)^{\frac N2}
\end{equation}
with \(A>0\), \(\rho>0\), and \(x_0\in\R^N\). Then, after a translation, there exist constants \(R>0\), \(c_1,c_2>0\), and \(C>0\) such that, for \(|x|\ge R\),
\begin{equation}
\label{eq3.15}c_1|x|^{-N}\le U(x)\le c_2|x|^{-N},
\end{equation}
\begin{equation}
\label{eq3.16}|(-\Delta)^sU(x)|
\le
C|x|^{-N-2s}(1+\log |x|),
\end{equation}
and
\begin{equation}
\label{eq3.17}\big(I_\alpha*U^{t_c}\big)(x)U(x)^{t_c-1}
\le
C|x|^{-N-2s}.
\end{equation}
\end{lemma}

\begin{proof}
By translation we may assume that \(x_0=0\). Throughout the proof, the
constants may depend on \(A,\rho,N,s,\alpha\), but not on \(x\). The explicit
formula \eqref{eq3.14} gives
\[
U(y)\le C(1+|y|)^{-N},
\qquad
|D^2U(y)|\le C(1+|y|)^{-N-2},
\qquad y\in\R^N.
\]
It also gives \eqref{eq3.15} for sufficiently large \(|x|\).

We first prove \eqref{eq3.16}. Put \(r=|x|\ge2\). The symmetric
singular integral formula gives
\[
|(-\Delta)^sU(x)|
\le C\int_{\R^N}
\frac{|2U(x)-U(x+z)-U(x-z)|}{|z|^{N+2s}}\,dz.
\]
If \(|z|\le r/2\), then \(|x\pm\theta z|\ge r/2\) for
\(0\le\theta\le1\). Taylor's formula therefore yields
\[
|2U(x)-U(x+z)-U(x-z)|\le Cr^{-N-2}|z|^2.
\]
Since \(s<1\), it follows that
\[
\begin{aligned}
\int_{|z|\le r/2}
\frac{|2U(x)-U(x+z)-U(x-z)|}{|z|^{N+2s}}\,dz
&\le Cr^{-N-2}\int_0^{r/2}t^{1-2s}\,dt\\
&\le Cr^{-N-2s}.
\end{aligned}
\]
For \(|z|>r/2\), the contribution of \(U(x)\) satisfies
\[
U(x)\int_{|z|>r/2}|z|^{-N-2s}\,dz\le Cr^{-N-2s}.
\]
Changing variables \(y=x+z\), we estimate the contribution of \(U(x+z)\)
by splitting into \(|y|\le2r\) and \(|y|>2r\). On the first region,
\[
\begin{aligned}
\int_{\substack{|x-y|>r/2\\ |y|\le2r}}
\frac{U(y)}{|x-y|^{N+2s}}\,dy
&\le Cr^{-N-2s}\int_{|y|\le2r}(1+|y|)^{-N}\,dy\\
&\le Cr^{-N-2s}(1+\log r).
\end{aligned}
\]
Here the logarithm follows from
\[
\int_{|y|\le2r}(1+|y|)^{-N}\,dy
\le C\left(1+\int_1^{2r}\frac{dt}{t}\right).
\]
On the second region, \(|x-y|\ge |y|/2\), so
\[
\int_{|y|>2r}\frac{U(y)}{|x-y|^{N+2s}}\,dy
\le C\int_{2r}^{\infty}t^{-N-2s-1}\,dt
\le Cr^{-N-2s}.
\]
The same estimates hold for \(U(x-z)\). Combining these bounds proves
\eqref{eq3.16}.

We next prove \eqref{eq3.17}. Since \(Nt_c=N+\alpha+2s\),
\[
U(y)^{t_c}\le C(1+|y|)^{-N-\alpha-2s},
\qquad U^{t_c}\in L^1(\R^N).
\]
We split the integral defining \(I_\alpha*U^{t_c}\) into the three regions
\[
\begin{gathered}
|y|\le r/2,\\
|y|>r/2,\quad |x-y|\le r/2,\\
|y|>r/2,\quad |x-y|>r/2.
\end{gathered}
\]
On the first region, \(|x-y|\ge r/2\), and hence
\[
\int_{|y|\le r/2}\frac{U(y)^{t_c}}{|x-y|^{N-\alpha}}\,dy
\le Cr^{-(N-\alpha)}\norm{U^{t_c}}_1
\le Cr^{-(N-\alpha)}.
\]
On the second region, using \(\alpha>0\), we obtain
\[
\begin{aligned}
\int_{\substack{|y|>r/2\\ |x-y|\le r/2}}
\frac{U(y)^{t_c}}{|x-y|^{N-\alpha}}\,dy
&\le Cr^{-N-\alpha-2s}
\int_{|z|\le r/2}|z|^{-(N-\alpha)}\,dz\\
&\le Cr^{-N-2s}.
\end{aligned}
\]
On the third region,
\[
\begin{aligned}
\int_{\substack{|y|>r/2\\ |x-y|>r/2}}
\frac{U(y)^{t_c}}{|x-y|^{N-\alpha}}\,dy
&\le Cr^{-(N-\alpha)}
\int_{|y|>r/2}(1+|y|)^{-N-\alpha-2s}\,dy\\
&\le Cr^{-(N-\alpha)}r^{-\alpha-2s}
=Cr^{-N-2s}.
\end{aligned}
\]
Consequently,
\[
(I_\alpha*U^{t_c})(x)\le Cr^{-(N-\alpha)}.
\]
Finally, \(N(t_c-1)=\alpha+2s\) gives
\[
U(x)^{t_c-1}\le Cr^{-\alpha-2s}.
\]
Multiplication of the last two estimates proves \eqref{eq3.17}.
\end{proof}

\begin{lemma}
\label{Lem3.15}
There is no nonnegative \(u\in S(a_*)\) such that \(u\) is an optimizer for
\eqref{eq2.2} and
\[
\big(I_\alpha*u^{q_L}\big)u^{q_L-1}
=
\kappa u
\qquad
\text{in }H^{-s}(\R^N)
\]
for some \(\kappa>0\).
\end{lemma}

\begin{proof}
Suppose, for a contradiction, that such a function \(u\) exists. By Lemma \ref{Lem3.13}, \(u\) has the explicit form \eqref{eq3.13}. Hence Lemma \ref{Lem3.14} applies to \(u\).

Since \(u\) is an optimizer for \eqref{eq2.2}, Lemma \ref{Lem3.10} gives
\begin{equation}
\label{eq3.18}(-\Delta)^s u+\beta u
=
\big(I_\alpha*u^{t_c}\big)u^{t_c-1}
\qquad
\text{in }\R^N,
\end{equation}
where
\[
\beta=
\frac{t_c-1}{a_*}
\norm{(-\Delta)^{s/2}u}_2^2>0.
\]
The equality is first understood in the distributional sense. Since \(u\) has
the smooth profile \eqref{eq3.13}, all terms in \eqref{eq3.18} are pointwise
well defined away from a bounded set, and the identity holds a.e. there.

By Lemma \ref{Lem3.14}, for large \(|x|\),
\[
|(-\Delta)^su(x)|
\le
C|x|^{-N-2s}(1+\log |x|)
\]
and
\[
\big(I_\alpha*u^{t_c}\big)(x)u(x)^{t_c-1}
\le
C|x|^{-N-2s}.
\]
Therefore \eqref{eq3.18} implies
\[
\beta u(x)
\le
C|x|^{-N-2s}(1+\log |x|)
\qquad
\text{for large }|x|.
\]
This contradicts the lower bound
\[
u(x)\ge c_1|x|^{-N}
\]
from \eqref{eq3.15}, because
\[
|x|^{-N-2s}(1+\log |x|)=o(|x|^{-N})
\qquad
\text{as }|x|\to+\infty.
\]
The contradiction proves the lemma.
\end{proof}

\begin{proposition}
\label{Pro3.16}
At the critical mass \(a=a_*\), the infimum \(m(a_*)\) is not attained on \(S(a_*)\), and problem \eqref{eq1.3} has no normalized solution with mass \(a_*\).
\end{proposition}

\begin{proof}
We first prove nonattainment. Suppose, for a contradiction, that \(m(a_*)\) is attained by some \(u\in S(a_*)\). By Proposition \ref{Pro3.3}, \(u\) attains equality in both sharp inequalities \eqref{eq2.1} and \eqref{eq2.2}. By Lemma \ref{Lem3.12}, replacing \(u\) by \(-u\) if necessary, we may assume that \(u\ge0\). Since \(u\) also attains equality in the lower endpoint inequality \eqref{eq2.1}, it is a critical point of the quotient
\[
\frac{\Dcal_{q_L}(w)}{\norm{w}_2^{2q_L}},
\qquad
w\in L^2(\R^N)\setminus\{0\}.
\]
Taking the first variation of this quotient at \(u\), we obtain
\[
\int_{\R^N}\big(I_\alpha*u^{q_L}\big)u^{q_L-1}v\,dx
=
\frac{\Dcal_{q_L}(u)}{\norm{u}_2^2}
\int_{\R^N}uv\,dx
\]
for all \(v\in C_c^\infty(\R^N)\). Hence
\[
\big(I_\alpha*u^{q_L}\big)u^{q_L-1}
=
\kappa u
\qquad
\text{in }H^{-s}(\R^N),
\]
where
\[
\kappa=\frac{\Dcal_{q_L}(u)}{\norm{u}_2^2}>0.
\]
This contradicts Lemma \ref{Lem3.15}. Hence \(m(a_*)\) is not attained.

We next prove nonexistence. Suppose, for a contradiction, that \eqref{eq1.3} has a normalized solution \(u\in S(a_*)\). By Proposition \ref{Pro3.9}, \(u\) is an optimizer for \eqref{eq2.2}. By Lemma \ref{Lem3.12}, replacing \(u\) by \(-u\) if necessary, we may assume that \(u\ge0\). Proposition \ref{Pro3.11} then gives
\[
\big(I_\alpha*u^{q_L}\big)u^{q_L-1}
=
\kappa u
\]
for some \(\kappa>0\). Again this contradicts Lemma \ref{Lem3.15}. Therefore no normalized solution with mass \(a_*\) exists.
\end{proof}

\section{Large masses and the Pohozaev constraint}
\label{sec:large}

\begin{proposition}
\label{Pro4.1}
If \(a>a_*\), then
\begin{equation*}
\inf_{u\in S(a)}J(u)=-\infty.
\end{equation*}
\end{proposition}

\begin{proof}
Let \(Q\in S(a_*)\) be an optimizer for the sharp inequality \eqref{eq2.2}. Then
\[
\Dcal_{t_c}(Q)
=
C_{N,\alpha,s}a_*^{t_c-1}
\norm{(-\Delta)^{s/2}Q}_2^2
=
t_c\norm{(-\Delta)^{s/2}Q}_2^2.
\]
Set
\[
v=
\sqrt{\frac{a}{a_*}}\,Q.
\]
Then \(v\in S(a)\). Moreover,
\[
\norm{(-\Delta)^{s/2}v}_2^2
=
\frac{a}{a_*}
\norm{(-\Delta)^{s/2}Q}_2^2
\]
and
\[
\Dcal_{t_c}(v)
=
\left(\frac{a}{a_*}\right)^{t_c}
\Dcal_{t_c}(Q).
\]
Therefore
\[
\begin{aligned}
P(v)
&=
\norm{(-\Delta)^{s/2}v}_2^2
-
\frac1{t_c}\Dcal_{t_c}(v)  \\
&=
\frac{a}{a_*}
\norm{(-\Delta)^{s/2}Q}_2^2
-
\left(\frac{a}{a_*}\right)^{t_c}
\norm{(-\Delta)^{s/2}Q}_2^2  \\
&=
\frac{a}{a_*}
\left[
1-
\left(\frac{a}{a_*}\right)^{t_c-1}
\right]
\norm{(-\Delta)^{s/2}Q}_2^2
<0,
\end{aligned}
\]
because \(a>a_*\).

For \(\sigma>0\), define
\[
v_\sigma(x)=\sigma^{\frac N2}v(\sigma x).
\]
Then \(v_\sigma\in S(a)\). By Lemma \ref{Lem2.4},
\[
J(v_\sigma)
=
\frac{\sigma^{2s}}2 P(v)
-
\frac1{2q_L}\Dcal_{q_L}(v).
\]
Since \(P(v)<0\), we obtain
\[
J(v_\sigma)\to-\infty
\qquad
\text{as }\sigma\to+\infty.
\]
Hence
\[
\inf_{u\in S(a)}J(u)=-\infty.
\]
\end{proof}

\begin{proposition}
\label{Pro4.2}
If \(a>a_*\), then the set
\begin{equation*}
\Pcal_a=
\left\{
 u\in S(a):P(u)=0
\right\}
\end{equation*}
is nonempty. Every normalized solution with mass \(a\) belongs to \(\Pcal_a\).
\end{proposition}

\begin{proof}
By Lemma \ref{Lem2.3} and the proof of Proposition \ref{Pro4.1}, there
exists \(v\in S(a)\), \(v\ge0\), such that \(P(v)<0\). We construct a
nonnegative \(w\in S(a)\) with \(P(w)>0\).

Choose \(\phi\in C_c^\infty(\R^N)\), \(\phi\ge0\), with
\(\norm{\phi}_2=1\), and fix \(r_0>0\) such that
\(\operatorname{supp}\phi\subset B_{r_0}(0)\). For an integer \(k\ge1\)
and \(R>4r_0\), let \(e_1=(1,0,\ldots,0)\) and set
\[
\phi_j^R(x)=\phi(x-jRe_1),
\qquad
w_{k,R}(x)=\sqrt{\frac ak}\sum_{j=1}^k\phi_j^R(x).
\]
The supports are pairwise disjoint. Thus
\[
\norm{w_{k,R}}_2^2
=\frac ak\sum_{j=1}^k\norm{\phi_j^R}_2^2=a,
\]
so \(w_{k,R}\in S(a)\).

We estimate the interactions between distinct translates. Denote by
\(c_{N,s}>0\) the constant in the singular integral definition of
\(( -\Delta)^s\). For \(i\ne j\), the bilinear Gagliardo identity and the
disjointness of the supports give
\[
\begin{aligned}
\int_{\R^N}(-\Delta)^{s/2}\phi_i^R
(-\Delta)^{s/2}\phi_j^R\,dx
&=-c_{N,s}\iint_{\R^N\times\R^N}
\frac{\phi_i^R(x)\phi_j^R(y)}{|x-y|^{N+2s}}\,dx\,dy.
\end{aligned}
\]
On these supports,
\[
|x-y|\ge |i-j|R-2r_0\ge R/2.
\]
It follows that
\[
\left|\int_{\R^N}(-\Delta)^{s/2}\phi_i^R
(-\Delta)^{s/2}\phi_j^R\,dx\right|
\le CR^{-N-2s}\norm{\phi}_1^2.
\]
Consequently, for each fixed \(k\),
\[
\norm{(-\Delta)^{s/2}w_{k,R}}_2^2
=a\norm{(-\Delta)^{s/2}\phi}_2^2
+O_k(R^{-N-2s})
\qquad\text{as }R\to\infty.
\]
Here and below, the constants in \(O_k\) may depend on \(k\), as well as
on the fixed parameters, but not on \(R\).

For the Choquard term, disjointness gives
\[
|w_{k,R}|^{t_c}
=\left(\frac ak\right)^{t_c/2}
\sum_{j=1}^k(\phi_j^R)^{t_c}.
\]
If \(i\ne j\), then
\[
\begin{aligned}
\int_{\R^N}
\big(I_\alpha*(\phi_i^R)^{t_c}\big)(\phi_j^R)^{t_c}\,dx
&\le CR^{-(N-\alpha)}\norm{\phi^{t_c}}_1^2.
\end{aligned}
\]
The diagonal terms are invariant under translation. Therefore
\[
\Dcal_{t_c}(w_{k,R})
=a^{t_c}k^{1-t_c}\Dcal_{t_c}(\phi)
+O_k(R^{-(N-\alpha)}).
\]
Combining the preceding estimates, for fixed \(k\) we have
\[
P(w_{k,R})\longrightarrow
 a\norm{(-\Delta)^{s/2}\phi}_2^2
 -\frac{a^{t_c}}{t_c}k^{1-t_c}\Dcal_{t_c}(\phi)
\qquad\text{as }R\to\infty.
\]
Since \(t_c>1\) and \(\norm{(-\Delta)^{s/2}\phi}_2^2>0\), we first choose
\(k\) so large that
\[
\frac{a^{t_c}}{t_c}k^{1-t_c}\Dcal_{t_c}(\phi)
\le\frac a2\norm{(-\Delta)^{s/2}\phi}_2^2.
\]
Keeping this \(k\) fixed, we then choose \(R\) sufficiently large to obtain
\[
P(w_{k,R})\ge\frac a4\norm{(-\Delta)^{s/2}\phi}_2^2>0.
\]
Set \(w=w_{k,R}\).

Both \(w\) and \(v\) are nonnegative and belong to \(S(a)\). Hence
\[
\norm{(1-t)w+tv}_2^2
\ge a\big((1-t)^2+t^2\big)\ge\frac a2,
\qquad 0\le t\le1.
\]
Thus
\[
\eta(t)=\sqrt a\,
\frac{(1-t)w+tv}{\norm{(1-t)w+tv}_2},
\qquad 0\le t\le1,
\]
is a continuous path in \(S(a)\) from \(w\) to \(v\). Since \(P\) is
continuous, \(P(w)>0\), and \(P(v)<0\), the intermediate value theorem
gives \(t_0\in(0,1)\) such that \(P(\eta(t_0))=0\). Therefore
\(\Pcal_a\ne\emptyset\). The last assertion follows from
Lemma \ref{Lem3.5}.
\end{proof}

\section{Proofs of the main theorems}\label{sec:proofs}
\begin{proof}[Proof of Theorem \ref{Thm1.1}]
Let \(0<a<a_*\). By Proposition \ref{Pro3.3},
\[
m(a)
=
-\frac1{2q_L}L_{N,\alpha}a^{q_L},
\]
and \(m(a)\) is not attained on \(S(a)\). Moreover, Proposition \ref{Pro3.6} shows that problem \eqref{eq1.3} has no normalized solution with mass \(a\). This proves Theorem \ref{Thm1.1}.
\end{proof}

\begin{proof}[Proof of Theorem \ref{Thm1.2}]
By Proposition \ref{Pro3.3},
\[
m(a_*)
=
-\frac1{2q_L}L_{N,\alpha}a_*^{q_L}.
\]
Moreover, Proposition \ref{Pro3.16} shows that \(m(a_*)\) is not attained on \(S(a_*)\), and that problem \eqref{eq1.3} has no normalized solution with mass \(a_*\). This proves Theorem \ref{Thm1.2}.
\end{proof}

\begin{proof}[Proof of Theorem \ref{Thm1.3}]
Let \(a>a_*\). Proposition \ref{Pro4.1} gives
\[
\inf_{u\in S(a)}J(u)=-\infty.
\]
Moreover, Proposition \ref{Pro4.2} shows that
\[
\Pcal_a
=
\{u\in S(a):P(u)=0\}
\]
is nonempty and that every normalized solution with mass \(a\) belongs to \(\Pcal_a\). Finally, if \(u\in S(a)\) is a normalized solution of \eqref{eq1.3}, then the Pohozaev identity follows from Lemma \ref{Lem3.5}, while the identity for the Lagrange multiplier and the sign condition
\[
\lambda<0
\]
follow from Lemma \ref{Lem3.8}. This proves Theorem \ref{Thm1.3}.
\end{proof}

\section*{Acknowledgments}

Y. Chen was supported by the National Natural Science Foundation of China (12161007) and Guangxi Natural Science Foundation Project (2023GXNSFAA026190).
Z. Yang was supported by the National Natural Science Foundation of China (12301145, 12261107), Yunnan Fundamental Research Projects (202201AU070031, 202401AU070123). 

\noindent{\bf Data availability:}  Data sharing is not applicable to this article as no new data were created or analyzed in this study.
	
\noindent{\bf Conflict of Interest:} The authors declare that they have no conflict of interest.


\begin{thebibliography}{99}

\bibitem{Ambrosio2025Remarks}
V. Ambrosio,
\newblock Remarks on the nonlinear fractional Choquard equation,
\newblock {\it Fract. Calc. Appl. Anal.} {\bf 28} (2025), 2241--2301.

\bibitem{BellazziniJeanjeanLuo2013}
J. Bellazzini, L. Jeanjean and T. Luo,
\newblock Existence and instability of standing waves with prescribed norm for a class of Schr\"odinger--Poisson equations,
\newblock {\it Proc. Lond. Math. Soc. (3)} {\bf 107} (2013), no. 2, 303--339.

\bibitem{ChenKumarYangZhang2026}
S. Chen, V. Kumar, Z. Yang and X. Zhang,
\newblock Normalized solutions for a class of fractional Choquard equations with the HLS lower critical term and a nonlocal perturbation,
\newblock preprint, arXiv:2604.12774, 2026.

\bibitem{ChenLiOu2006}
W. Chen, C. Li and B. Ou,
\newblock Classification of solutions for an integral equation,
\newblock {\it Comm. Pure Appl. Math.} {\bf 59} (2006), 330--343.

\bibitem{ChenYangTJM2025}
Z. Chen and Y. Yang,
\newblock Normalized solutions for the fractional Choquard equations with lower critical exponent and nonlocal perturbation,
\newblock {\it Taiwanese J. Math.} {\bf 29} (2025), no. 2, 261--294.

\bibitem{CingolaniGalloTanaka2024}
S. Cingolani, M. Gallo and K. Tanaka,
\newblock Infinitely many free or prescribed mass solutions for fractional Hartree equations and Pohozaev identities,
\newblock {\it Adv. Nonlinear Stud.} {\bf 24} (2024), no. 2, 303--334.

\bibitem{CingolaniJeanjean2019}
S. Cingolani and L. Jeanjean,
\newblock Stationary waves with prescribed \(L^2\)-norm for the planar Schr\"odinger--Poisson system,
\newblock {\it SIAM J. Math. Anal.} {\bf 51} (2019), no. 4, 3533--3568.

\bibitem{DAveniaSicilianoSquassina2015}
P. d'Avenia, G. Siciliano and M. Squassina,
\newblock On fractional Choquard equations,
\newblock {\it Math. Models Methods Appl. Sci.} {\bf 25} (2015), no. 8, 1447--1476.

\bibitem{FengHeMeng2023}
Z. Feng, X. He and Y. Meng,
\newblock Normalized solutions of fractional Choquard equation with critical nonlinearity,
\newblock {\it Differential Integral Equations} {\bf 36} (2023), no. 7--8, 593--620.

\bibitem{HeRadulescuZou2022}
X. He, V. D. R\u adulescu and W. Zou,
\newblock Normalized ground states for the critical fractional Choquard equation with a local perturbation,
\newblock {\it J. Geom. Anal.} {\bf 32} (2022), no. 10, Paper No. 252, 51 pp.

\bibitem{Jeanjean1997}
L. Jeanjean,
\newblock Existence of solutions with prescribed norm for semilinear elliptic equations,
\newblock {\it Nonlinear Anal.} {\bf 28} (1997), 1633--1659.

\bibitem{LanHeMeng2023}
J. Lan, X. He and Y. Meng,
\newblock Normalized solutions for a critical fractional Choquard equation with a nonlocal perturbation,
\newblock {\it Adv. Nonlinear Anal.} {\bf 12} (2023), no. 1, Paper No. 20230112, 40 pp.

\bibitem{LiLuo2020}
G. Li and X. Luo,
\newblock Existence and multiplicity of normalized solutions for a class of fractional Choquard equations,
\newblock {\it Sci. China Math.} {\bf 63} (2020), no. 3, 539--558.

\bibitem{LiYe2014}
G. Li and H. Ye,
\newblock The existence of positive solutions with prescribed \(L^2\)-norm for nonlinear Choquard equations,
\newblock {\it J. Math. Phys.} {\bf 55} (2014), 121501, 19 pp.

\bibitem{Lieb1977}
E. H. Lieb,
\newblock Existence and uniqueness of the minimizing solution of Choquard's nonlinear equation,
\newblock {\it Stud. Appl. Math.} {\bf 57} (1977), 93--105.

\bibitem{Lieb1983}
E. H. Lieb,
\newblock Sharp constants in the Hardy--Littlewood--Sobolev and related inequalities,
\newblock {\it Ann. of Math.} {\bf 118} (1983), 349--374.

\bibitem{LiebLoss2001}
E. H. Lieb and M. Loss,
\newblock {\it Analysis},
\newblock Graduate Studies in Mathematics, Vol. 14, American Mathematical Society, Providence, RI, 2001.

\bibitem{Lions1982}
P.-L. Lions,
\newblock Sym\'etrie et compacit\'e dans les espaces de Sobolev,
\newblock {\it J. Funct. Anal.} {\bf 49} (1982), 315--334.

\bibitem{MorozVanSchaftingen2013}
V. Moroz and J. Van Schaftingen,
\newblock Groundstates of nonlinear Choquard equations: existence, qualitative properties and decay asymptotics,
\newblock {\it J. Funct. Anal.} {\bf 265} (2013), 153--184.

\bibitem{MorozVanSchaftingen2015}
V. Moroz and J. Van Schaftingen,
\newblock Existence of groundstates for a class of nonlinear Choquard equations,
\newblock {\it Trans. Amer. Math. Soc.} {\bf 367} (2015), 6557--6579.

\bibitem{PalatucciPisante2014}
G. Palatucci and A. Pisante,
\newblock Improved Sobolev embeddings, profile decomposition, and concentration-compactness for fractional Sobolev spaces,
\newblock {\it Calc. Var. Partial Differential Equations} {\bf 50} (2014), 799--829.

\bibitem{YuanChenTang2020}
S. Yuan, S. Chen and X. Tang,
\newblock Normalized solutions for Choquard equations with general nonlinearities,
\newblock {\it Electron. Res. Arch.} {\bf 28} (2020), no. 1, 291--309.

\end{thebibliography}
\end{document}